\documentclass[11pt,twoside]{article}
\usepackage{mathrsfs}
\usepackage{amsmath}
\usepackage{amsthm}
\usepackage{amsfonts}
\usepackage{amssymb}
\usepackage{latexsym}
\usepackage[all]{xy}

\date{\empty}
\allowdisplaybreaks[4] 

\numberwithin{equation}{section} \theoremstyle{plain}
\newtheorem*{thm*}{Main Theorem}
\newtheorem{theorem}{Theorem}[section]
\newtheorem{corollary}[theorem]{Corollary}
\newtheorem*{corollary*}{Corollary}

\newtheorem*{claim*}{Claim}
\newtheorem{lemma}[theorem]{Lemma}
\newtheorem*{lemma*}{Lemma}
\newtheorem{proposition}[theorem]{Proposition}
\newtheorem*{proposition*}{Proposition}
\newtheorem{remark}[theorem]{Remark}
\newtheorem*{remark*}{Remark}
\newtheorem{example}[theorem]{Example}
\newtheorem*{example*}{Example}

\newtheorem*{question*}{Question}
\newtheorem{definition}[theorem]{Definition}
\newtheorem*{definition*}{Definition}

\newtheorem*{acknowledgements*}{ACKNOWLEDGEMENTS}

\newcommand{\core}[1]{#1^{\tiny{\textcircled{\tiny\#}}}}
\newcommand{\cored}[1]{#1_{\tiny{\textcircled{\tiny\#}}}}
\newcommand{\pn}{\overline{p}}

\begin{document}
\begin{center}
{\large  \bf EP elements in rings with involution}\\
\vspace{0.8cm} {\small \bf Sanzhang  Xu and Jianlong Chen}
\footnote{ E-mail: xusanzhang5222@126.com.
Corresponding author,
E-mail: jlchen@seu.edu.cn
}

\vspace{0.6cm} {\rm Department of Mathematics, Southeast University, Nanjing 210096, China}

\vspace{0.8cm} {\small \bf Julio Ben\'{\i}tez}
\footnote{E-mail: jbenitez@mat.upv.es}

\vspace{0.6cm} {\rm Universidad Polit\'{e}cnica de Valencia, Instituto de Matem\'{a}tica Multidisciplinar, Valencia, 46022,
Spain}
\end{center}

\bigskip

{ \bf  Abstract:}  \leftskip0truemm\rightskip0truemm Let $R$ be a unital ring with involution.
We first show that the EP elements in $R$ can be characterized by three equations.
Namely, let $a\in R$, then
$a$ is EP if and only if there exists $x\in R$ such that $(xa)^{\ast}=xa$, $xa^{2}=a$ and $ax^{2}=x.$
Any EP element in $R$ is core invertible and Moore-Penrose invertible.
We give more equivalent conditions for a core (Moore-Penrose) invertible element to be an EP element.
Finally, any EP element is characterized in terms of the $n$-EP property, which is a generalization of the bi-EP property.


 \bigskip

\section { \bf Introduction}

Throughout this paper, $R$ will denote a unital ring with involution, i.e.,
a ring with a mapping $a\mapsto a^*$
satisfying $(a^*)^*=a$, $(ab)^*=b^*a^*$ and $(a+b)^*=a^*+b^*$, for all $a,b\in R$.
The notion of core inverse for a complex matrix was introduced by Baksalary and Trenkler \cite{BT}.
In \cite{RDD}, Raki\'{c} et al. generalized the core inverse of a complex matrix to the case of an element in $R$.
More precisely, let $a,x\in R$, if
$$
axa=a,~xR=aR\quad\text{and}\quad Rx=Ra^{\ast},
$$
then $x$ is called a {\em core inverse} of $a$. If such an element $x$ exists, then it is
unique and denoted by $\core{a}$. The set of all core invertible elements in $R$ will be
denoted by $\core{R}$.
Also, in \cite{RDD} the authors defined a related inner inverse in a ring with an involution.
If $a \in R$, then $x \in R$ is called a {\em dual core inverse} of $a$ if
$$
a x a = a, \ x R = a^* R \quad \text{and} \quad R x = Ra.
$$
If such an element $x$ exists, then it is
unique and denoted by $\cored{a}$. The set of all dual core invertible elements in $R$ will be
denoted by $\cored{R}$.
It is elemental to prove that $a \in \core{R}$ if an only if $a^* \in \cored{R}$, and
in this case, one has $(\core{a})^* = \cored{(a^*)}$. This last observation permits
to get results concerning dual core inverses from the corresponding
results on core inverses.

Let $a,x\in R$. If
$$
axa=a, \quad xax=x, \quad (ax)^{\ast}=ax \quad\text{and}\quad (xa)^{\ast}=xa,
$$
then $x$ is called a {\em Moore-Penrose inverse} of $a$.
If such an element $x$ exists, then it is unique and denoted by $a^{\dagger}$.
The set of all Moore-Penrose invertible elements will be denoted by $R^{\dagger}$.

Let $a\in R$. It can be easily proved that the set of elements $x\in R$ such that
$$
axa=a, \quad xax=x\quad\text{and}\quad ax=xa
$$
is empty or a singleton. If this set is a singleton, its unique element is called the
{\em group inverse} of $a$ and denoted by $a^\#$.
The set of all group invertible elements will be denoted by $R^\#$.
The subset of $R$ composed of all invertible elements will be denoted by $R^{-1}$.

A matrix $A\in \mathbb{C}^{n\times n}$ is called an {\em EP} (range-Hermitian) matrix if $\mathcal{R}(A)=\mathcal{R}(A^{\ast})$, where
$\mathbb{C}^{n\times n}$ denotes the set of all $n\times n$ matrices over the field of complex numbers and $\mathcal{R}(A)$ stands for the range (column space)
of $A\in \mathbb{C}^{n\times n}$. This concept was first introduced by Schwerdtfeger in \cite{HS}.
An element $a\in R$ is said to be an {\em EP} element if $a\in R^{\dagger}\cap R^\#$ and $a^{\dagger}=a^\#$ \cite{H}. The set of all EP elements will be denoted by $R^{\mathrm{EP}}$.
Mosi\'{c} et al. in \cite[Theorem 2.1]{MDK} gave several equivalent conditions such that an element in $R$ to be an EP element.
Patr\'{i}cio and Puystjens in \cite[Proposition 2]{PP} proved that for a Moore-Penrose invertible element $a\in R$, $a\in R^{\mathrm{EP}}$ if and only if $aR=a^\ast R$. As for a Moore-Penrose invertible element $a\in R$, $a\in R^{\mathrm{EP}}$ if and only if $aa^{\dagger}=a^{\dagger}a$, thus we deduce that
$aa^{\dagger}=a^{\dagger}a$ if and only if $aR=a^\ast R$.
In \cite[Theorem 3.1]{RDD}, Raki\'{c} et al. investigated some equivalent conditions such that a (dual) core invertible element in $R$ to be an EP element. Also, they showed that
$R^{\dagger}\cap \core{R} = R^{\dagger}\cap R^\#$.
Motivated by \cite{BO, MDK, PP, RDD}, in this paper, we will give new equivalent characterizations such that an element in $R$ to be an EP element.

We first show that the EP elements in $R$ can be characterized by three equations.
That is, let $a\in R$, then
$a\in R^{\mathrm{EP}}$ if and only if there exists $x\in R$ such that $(xa)^{\ast}=xa$, $xa^{2}=a$ and $ax^{2}=x.$
In \cite{RDD}, Raki\'{c} et al. proved that $a\in R^{\dagger}$ if and only if there exists $x\in R$
such that $axa=a$, $xR=a^{\ast}R$ and $Rx=Ra^{\ast}$. Inspired by this result, we show that $a\in R^{\mathrm{EP}}$ if
and only if there exists $x\in R$ such that
$$
axa=a,~xR=aR\quad\text{and}\quad Rx^{\ast}=Ra.
$$
In \cite[Theorem 16]{BO}, for an operator $T\in L(X)$, where $X$ is a Banach space, Boasso proved that
for a Moore-Penrose invertible operator $T$,
$T$ is an EP operator if and only there exists an invertible operator
$P\in L(X)$ such that $T^{\dagger}=PT$.
We generalize this result to the ring case. Moreover, for $a\in R^{\dagger}$, we show that
$a\in R^{\mathrm{EP}}$ if and only if there exists a (left) invertible element $v$ such that $a^{\dagger}=va$.
Similarly, for $a\in \core{R}$, then $a\in R^{\mathrm{EP}}$ if and only if there exists
a (left) invertible element $s$ such that $\core{a}=sa$.

In \cite{RDD}, Raki\'{c} et al. proved that $a\in R^{\mathrm{EP}}$ if and only if
$a\in R^{\dagger}\cap R^\#$ with $a^{\dagger}=\core{a}$. Also, it is proved that
$a\in R^{\mathrm{EP}}$ if and only if $a\in \core{R}$ with $a^\#=\core{a}$.
In \cite[Theorem 2.1]{MD}, Mosi\'{c} and Djordjevi\'{c} proved that
$a\in R^{\mathrm{EP}}$ if and only if $a\in R^\#\cap R^{\dagger}$ with $a^{n}a^{\dagger}=a^{\dagger}a^{n}$ for all choices $n\geqslant 1$.
This result also can be found in \cite[Theorem 2.4]{C} by Chen.
Motivated by \cite{MD,RDD}, we will give more new equivalent conditions under which a core invertible element is an EP element.
And we define the concept of $n$-EP as a generalization of bi-EP. As a application, we will use $n$-EP property to give an equivalent characterization
of the EP elements in $R$.
\section { \bf New characterizations of EP elements by equations}

In this section, we first show that any EP element in $R$ can be characterized by three equations.
Let us begin with an auxiliary lemma.

\begin{lemma} \emph{\cite[Theorem 7.3]{KP}} \label{group-aa-aab}
Let $a\in R$. Then $a\in R^{\mathrm{EP}}$ if and only if $a\in R^\#$ with $(aa^\#)^{\ast}=aa^\#$.
\end{lemma}

It is well known that the group inverse of an element in a ring can be characterized by three equations and
the Moore-Penrose inverse of an element in a ring can be characterized by four equations.
In the following theorem, we show that an EP element in a ring can be described by three equations.
\begin{theorem} \label{equations-a}
Let $a\in R$. Then $a\in R^{\mathrm{EP}}$ if and only if there exists $x\in R$ such that
\begin{equation}\label{eqth22}
(xa)^{\ast}=xa,~~xa^{2}=a~~\text{and}~~ax^{2}=x.
\end{equation}
\end{theorem}
\begin{proof}
Suppose $a\in R^{\mathrm{EP}}$. Let $x=a^{\dagger}=a^\#$, then $(xa)^{\ast}=(a^{\dagger}a)^{\ast}=a^{\dagger}a=xa,$
$xa^{2}=a^\#a^{2}=a$ and $ax^{2}=a(a^\#)^{2}=a^\#=x.$ Conversely, if there exists $x\in R$ such that
$(xa)^{\ast}=xa$, $xa^{2}=a$ and $ax^{2}=x,$ then $a(x^2a) = (ax^2)a = xa = x(xa^2) = (x^2a)a$,
$a(x^2a)a = (xa)a = xa^2 = a$, and
$(x^2a)a(x^2a) = (xa)(x^2a) = x(ax^2)a = x^2a$.
These three equalities prove that $a \in R^\#$, $a^\# = x^2a$, and $aa^\# = xa$. By
Lemma \ref{group-aa-aab}, we get $a \in R^{\rm EP}$.
\end{proof}

For an idempotent $p$ in a ring $R$,
every $a\in R$ can be written as
$$a = pap+pa(1-p)+(1-p)ap+(1-p)a(1-p)$$
or in the matrix form
$$
a = \left[ \begin{matrix}
a_{11} & a_{12} \\ a_{21} & a_{22}
\end{matrix} \right],
$$
where $a_{11}=pap$, $a_{12}=pa(1-p)$, $a_{21}=(1-p)ap$ and $a_{22}=(1-p)a(1-p)$.

Let us observe that $pRp$ and $(1-p)R(1-p)$ are rings whose unities are $p$ and
$1-p$, respectively. Also, we notice that if $p=p^{\ast}$, then the above matrix
representation preserves the involution. The term
{\em projection} will be reserved for a Hermitian idempotent.

Suppose in this paragraph that $a \in R$ is an EP element. If we denote
$p=aa^\dagger=a^\dag a$, since $ap=pa=a$ and $a^\dag p = p a^\dag = a^\dagger$, then
the matrix representations of $a$ and $a^\dagger$ with respect to the Hermitian idempotent $p$
are
\begin{equation}\label{mra}
a = \left[ \begin{array}{cc} a & 0 \\ 0 & 0 \end{array} \right] \quad \text{and} \quad
a^\dag = \left[ \begin{array}{cc} a^\dag & 0 \\ 0 & 0 \end{array} \right],
\end{equation}
respectively.

Recall that a ring $R$ is
{\em prime} if for any two elements $a$ and $b$ of $R$, $aRb = 0$ implies that either $a = 0$ or $b = 0$
and a ring $R$ is
{\em semiprime} if for any element $a$ in $R$, $aRa = 0$ implies that $a = 0$.

In next theorem, the set of elements $x \in R$ satisfying (\ref{eqth22})
is described.

\begin{theorem} \label{biao-b}
Let $a \in R$. If $a$ is EP, then $\{ x \in R: (xa)^* = xa, xa^2 =a, ax^2 = x \} =
\{a^\dag +aa^\dag y (1-aa^\dag): y \in R\}$. Moreover, if $R$ is a prime ring, then
$\{ x \in R: (xa)^* = xa, xa^2 =a, ax^2 = x \} = \{ a^\dag\}$ if and only if $a=0$ or
$a$ is invertible.
\end{theorem}
\begin{proof}
Suppose $a$ is an EP element. We use the matrix representations
of $a$ and $a^\dag$ with respect to the projection $p=aa^\dagger$ given in (\ref{mra}).
Let $x=\left[\begin{smallmatrix} u & v \\ w & z \end{smallmatrix} \right]$ be the representation
of any $x\in\{ x \in R: (xa)^* = xa, xa^2 =a, ax^2 = x \}$ with respect to $p$.
From $xa^{2}=a$, we get
$$
\left[\begin{matrix} a & 0 \\ 0 & 0 \end{matrix} \right]
=\left[\begin{matrix} u & v \\ w & z \end{matrix} \right]
\left[\begin{matrix} a^{2} & 0 \\ 0 & 0 \end{matrix} \right]
=\left[\begin{matrix} ua^{2} & 0 \\ wa^{2} & 0 \end{matrix} \right].
$$
Since $a$ is EP and $aa^{\dagger}=p=a^{\dagger}a$, then $a$ is invertible in $pRp$ and its
inverse is $a^{\dagger}$. Hence from $a=ua^{2}$ and $0=wa^{2}$, we obtain $u=a^{\dagger}$
and $0=w$, respectively. Now, from $x=ax^{2}$ we have
$$
\left[\begin{matrix} a^{\dagger} & v \\ 0 & z \end{matrix} \right]
= \left[\begin{matrix} a & 0 \\ 0 & 0 \end{matrix} \right]
\left[\begin{matrix} a^{\dagger} & v \\ 0 & z \end{matrix} \right]x
=\left[\begin{matrix} p & av \\ 0 & 0 \end{matrix} \right]
\left[\begin{matrix} a^{\dagger} & v \\ 0 & z \end{matrix} \right]
=\left[\begin{matrix} a^{\dagger} & v+avz \\ 0 & 0 \end{matrix} \right],
$$
which implies $z=0$. Therefore,
$$
x=\left[\begin{matrix} a^{\dagger} & v \\ 0 & 0 \end{matrix} \right]
=a^{\dagger}+v,
$$
that is $\{ x \in R: (xa)^* = xa, xa^2 =a, ax^2 = x \}\subseteq\{a^{\dagger}+aa^{\dagger}y(1-aa^{\dagger}):y\in R\}$.

Let us prove the opposite inclusion. We have that $aa^{\dagger}=a^{\dagger}a$ since $a$ is EP. Let
$x=a^{\dagger}+aa^{\dagger}y(1-aa^{\dagger})$, then
\begin{equation*}
\begin{split}
xa=& [a^{\dagger}+aa^{\dagger}y(1-aa^{\dagger})]a=a^{\dagger}a \text{ is Hermitian},\\
xa^{2}=& [a^{\dagger}+aa^{\dagger}y(1-aa^{\dagger})]a^{2}=a^{\dagger}a^{2}=a,\\
ax^{2}=& a[a^{\dagger}+aa^{\dagger}y(1-aa^{\dagger})]^{2}
=[aa^{\dagger}+ay(1-aa^{\dagger})][a^{\dagger}+aa^{\dagger}y(1-aa^{\dagger})]\\
=&a^{\dagger}+aa^{\dagger}y(1-aa^{\dagger})=x.
\end{split}
\end{equation*}
Suppose that $R$ is a prime ring. If $a=0$, then
$\{ x \in R: (xa)^* = xa, xa^2 =a, ax^2 = x \}=\{0\}$. If $a$ is invertible,
then $\{ x \in R: (xa)^* = xa, xa^2 =a, ax^2 = x \} = \{a^{-1}\}$.
If $\{ x \in R: (xa)^* = xa, xa^2 =a, ax^2 = x \}$ is a singleton,
then $aa^\dag y(1-aa^\dag)=0$ for all $y \in R$, by
using that $R$ is prime, then $aa^\dag=0$ or $1-aa^\dagger=0$. The first
of the previous alternatives is equivalent to $a=0$ and the second one (since $a$ is EP)
is equivalent to the invertibility of $a$.
\end{proof}

We will also use the following notations: $aR=\{ax : x\in R\}$, $Ra=\{xa : x\in R\}$,
${}^{\circ}a = \{x\in R : xa=0\}$ and $a^{\circ}=\{x\in R : ax=0\}$
.
The following lemma will be useful in the sequel.
\begin{lemma}\emph{\cite[Lemma 8]{VN}} \label{annihilator}
Let $a, b\in R$. Then:
\begin{itemize}
\item[{\rm (1)}] $aR\subseteq bR$ implies ${}^{\circ}b\subseteq {}^{\circ}a$ and the converse
is valid whenever $b$ is regular;
\item[{\rm (2)}] $Ra\subseteq Rb$ implies $b^{\circ}\subseteq a^{\circ}$ and the converse
is valid whenever $b$ is regular.
\end{itemize}
\end{lemma}

\begin{theorem} \label{equations-b}
Let $a\in R$. Then the following are equivalent:
\begin{itemize}
\item[{\rm (1)}] $a\in R^{\mathrm{EP}}$;
\item[{\rm (2)}] there exists $x\in R$ such that $axa=a$, $xR=aR$ and $Rx^{\ast}=Ra$;
\item[{\rm (3)}] there exists $x\in R$ such that $axa=a$, $xR=aR$ and $Rx^{\ast}\subseteq Ra$;
\item[{\rm (4)}] there exists $x\in R$ such that $xax=x$, $xR=aR$ and $Rx^{\ast}=Ra$;
\item[{\rm (5)}] there exists $x\in R$ such that $xax=x$, $xR=aR$ and $Ra\subseteq Rx^{\ast}$;
\item[{\rm (6)}] there exists $x\in R$ such that $axa=a$, ${}^{\circ}x={}^{\circ} a$ and
$(x^{\ast})^{\circ}=a^{\circ}$;
\item[{\rm (7)}] there exists $x\in R$ such that $axa=a$, ${}^{\circ}x={}^{\circ}a$
and $a^{\circ}\subseteq (x^{\ast})^{\circ}$;
\item[{\rm (8)}] there exists $x\in R$ such that $xax=x$, ${}^{\circ}x={}^{\circ} a$
and $(x^{\ast})^{\circ}=a^{\circ}$;
\item[{\rm (9)}] there exists $x\in R$ such that $xax=x$, ${}^{\circ}x={}^{\circ}a$ and
$(x^{\ast})^{\circ}\subseteq a^{\circ}$.
\end{itemize}
\end{theorem}
\begin{proof}
$(1)\Rightarrow(2)$: Let $x=a^{\dagger}=a^\#$, then $axa=a$, $x=a(a^\#)^{2}$,
$x^{\ast}=(a^{\dagger})^{\ast}=(a^{\dagger}aa^{\dagger})^{\ast}=(a^{\dagger})^{\ast}a^{\dagger}a,$
and $xa^{2}=a^\#a^{2}=a=aa^{\dagger}a=aa^{\ast}(a^{\dagger})^{\ast}=aa^{\ast}x^{\ast}.$ Thus $xR=aR$ and $Rx^{\ast}=Ra$.

$(2)\Rightarrow(3)$ and $(6)\Rightarrow(7)$ are clear.

$(2)\Rightarrow(6)$ and $(3)\Rightarrow(7)$ are obvious by Lemma \ref{annihilator}.

$(7)\Rightarrow(1)$: Suppose there exists $x\in R$ such that $axa=a$, $^{\circ}x=^{\circ}\!\!a$ and $a^{\circ}\subseteq (x^{\ast})^{\circ}$.
Since $(1-ax)a=0,$ then $1-ax\in$ ${}^{\circ}a={}^{\circ}x$, hence $(1-ax)x=0$.
Since $a(1-xa)=0$, then $1-xa\in a^{\circ}\subseteq (x^{\ast})^{\circ}$,
hence $x^{\ast}(1-xa)=0$, i.e., $x=(xa)^{\ast}x$.
We get $xa=(xa)^{\ast}xa$, hence $xa$ is Hermitian. Finally, $x=xax$ implies $1-xa\in$ $^{\circ}x={}^{\circ}a$,
whence $xa^{2}=a$.
Therefore $a\in R^{\mathrm{EP}}$ by Theorem \ref{equations-a}.

The implications $(1)\Leftrightarrow(4)\Leftrightarrow(5)\Leftrightarrow(8)\Rightarrow(9)$
are similar to  $(1)\Leftrightarrow(2)\Leftrightarrow(3)\Leftrightarrow(6)\Rightarrow(7)$.

$(9)\Rightarrow(1)$: There exists $x\in R$ such that $xax=x$, ${}^{\circ}x={}^{\circ}a$
and $(x^{\ast})^{\circ}\subseteq a^{\circ}$.
It is obvious that $(x^{\ast})^{\circ}\subseteq a^{\circ}$ is equivalent to
$^{\circ}x\subseteq {}^{\circ}(a^{\ast})$. We have $a=xa^{2}$ since $(1-xa)x=0$ implies $(1-xa)a=0$.
Similarly, we have $a^{\ast}=xaa^{\ast}$ since $(1-xa)x=0$ implies $(1-xa)a^{\ast}=0$.
Thus $(xa)^{\ast}=a^{\ast}x^{\ast}=xaa^{\ast}x^{\ast}=xa(xa)^{\ast}$, that is $(xa)^{\ast}=xa.$
By $a^{\ast}=xaa^{\ast}$ and $(xa)^{\ast}=xa$, we have $a^{\ast}=xaa^{\ast}=(xa)^{\ast}a^{\ast}=(axa)^{\ast}$, that is $a=axa.$
Hence by ${}^{\circ}x={}^{\circ}a$, we have $(1-ax)a=0$ implies $(1-ax)x=0$, which gives
$x=ax^{2}$. Therefore $a\in R^{\mathrm{EP}}$ by Theorem~\ref{equations-a}.
\end{proof}

\begin{theorem} \label{biao-c}
Let $a\in R^{\rm EP}$ and denote $p=aa^{\dagger}$. Then the following sets are the same
\begin{itemize}
\item[{\rm (1)}] $\{x\in R: axa=a, xR\subseteq aR\}=\{a^{\dagger}+py(1-p):y\in R\}$;
\item[{\rm (2)}] $\{x\in R: xax=x, xR=aR\}=\{a^{\dagger}+py(1-p):y\in R\}$;
\item[{\rm (3)}] $\{x\in R: axa=a, {}^{\circ}a\subseteq {}^{\circ}x\} =
\{a^{\dagger}+py(1-p):y\in R\}$;
\item[{\rm (4)}] $\{x\in R: xax=x, {}^{\circ}a = {}^{\circ}x\} =
\{a^{\dagger}+py(1-p):y \in R\}$.
\end{itemize}
Furthermore, if $R$ is prime, then any of the above subsets is a singleton if and only if
$a=0$ or $a$ is invertible.
\end{theorem}
\begin{proof}
We use the matrix representations
of $a$ and $a^\dag$ with respect to the projection $p=aa^\dagger$ given in (\ref{mra}).
Let $x=\left[\begin{smallmatrix} u & v \\ w & z \end{smallmatrix} \right]$ be the representation
of any $x$ with respect to $p$.

(1) Let $x$ satisfy $axa=a$ and $xR \subseteq aR$. From $axa=a$, we have
$$
\left[\begin{matrix} a & 0 \\ 0 & 0 \end{matrix} \right]
=\left[\begin{matrix} a & 0 \\ 0 & 0 \end{matrix} \right]
\left[\begin{matrix} u & v \\ w & z \end{matrix} \right]
\left[\begin{matrix} a & 0 \\ 0 & 0 \end{matrix} \right]
=\left[\begin{matrix} au & av \\ 0 & 0 \end{matrix} \right]
\left[\begin{matrix} a & 0 \\0 & 0 \end{matrix} \right]
=\left[\begin{matrix} aua & 0 \\ 0 & 0 \end{matrix} \right].
$$

Since $a\in R^{\mathrm{EP}}$, we have that $a$ is invertible in $pRp$ and its inverse is $a^{\dagger}.$ Hence $a=aua$ gives $u=a^{\dagger}.$
Since $xR \subseteq aR$, we can write
$$
\left[\begin{matrix} u & v \\ w & z \end{matrix} \right]
=\left[\begin{matrix} a & 0 \\ 0 & 0 \end{matrix} \right]
\left[\begin{matrix} \xi_{1} & \xi_{2} \\ \xi_{3} & \xi_{4} \end{matrix} \right]
=\left[\begin{matrix} a\xi_{1} & a\xi_{2} \\ 0 & 0 \end{matrix} \right].
$$
Therefore, $w=z=0$. Hence
$$
x=\left[\begin{matrix} a^{\dagger} & v \\ 0 & 0 \end{matrix} \right]
=a^{\dagger}+px(1-p)\in \{a^{\dagger}+py(1-p):y\in R\}.
$$
The opposite inclusion is trivial.

(2) Let $x \in R$ satisfy $xax=x$ and $xR = aR$.
Since $x \in a R$, by the proof of (1), we have $w=z=0$. Now,
since $a \in xR$, we can write
$$
\left[\begin{matrix} a & 0 \\ 0 & 0 \end{matrix} \right]
=\left[\begin{matrix} u & v \\ 0 & 0 \end{matrix} \right]
\left[\begin{matrix} \delta_{1} &  \delta_{2} \\ \delta_{3} &  \delta_{4} \end{matrix} \right],
$$
which implies
\begin{equation} \label{re-a}
a=u \delta_{1}+v \delta_{3},\qquad 0=u \delta_{2}+v \delta_{4}.
\end{equation}
Now, we use $xax=x$:
$$
\left[\begin{matrix} u & v \\ 0 & 0 \end{matrix} \right]
=\left[\begin{matrix} u & v \\ 0 & 0 \end{matrix} \right]
\left[\begin{matrix} a & 0 \\ 0 & 0 \end{matrix} \right]
\left[\begin{matrix} u & v \\ 0 & 0 \end{matrix} \right]
=\left[\begin{matrix} ua & 0 \\ 0 & 0 \end{matrix} \right]
\left[\begin{matrix} u & v \\ 0 & 0 \end{matrix} \right]
=\left[\begin{matrix} uau & uav \\ 0 & 0 \end{matrix} \right].
$$
Therefore
\begin{equation} \label{re-b}
u=uau,\qquad v=uav.
\end{equation}
Post-multiply the first equality of (\ref{re-b}) by $ \delta_{1}$ and the second equality of (\ref{re-b}) by $ \delta_{3}$ to obtain
\begin{equation*}
u \delta_{1}=uau \delta_{1}, \qquad v \delta_{3}=uav \delta_{3}.
\end{equation*}
From (\ref{re-a}),
\begin{equation*}
a=u \delta_{1}+v \delta_{3}=uau \delta_{1}+uav \delta_{3}=ua(u \delta_{1}+v \delta_{3})=ua^{2}.
\end{equation*}
We get $u=a^\dag$ because $a$ is invertible in $pRp$ and its inverse is $a^{\dagger}$. Therefore
$$
x=\left[\begin{matrix} a^{\dagger} & v \\ 0 & 0 \end{matrix}\right] = a^{\dagger}+v=a^{\dagger}+px(1-p)\in \{a^{\dagger}+py(1-p):y\in R\}.
$$
For the opposite inclusion, it is easy to check that
$[a^\dag + py(1-p)]a[a^\dag + py(1-p)] = a^\dag + py(1-p)$
and $a^\dag + py(1-p) \in a R$ in view of $a \in R^{\rm EP}$.
From $a =[a^\dag + py(1-p)]a^2$, we deduce that $a \in [a^\dag + py(1-p)]R$.

The proof of statements $(3)$ and $(4)$ follows from $(1)$ and $(2)$, respectively,
since by Lemma \ref{annihilator}, we obtain that $x\in aR$ is equivalent to
${}^{\circ}a\subseteq {}^{\circ}x$ and $xR=aR$ is equivalent to $^{\circ}a = {}^{\circ}x$,
respectively.

The proof of the last affirmation of this theorem has the same proof as the corresponding
part of Theorem \ref{biao-b}.
\end{proof}

By considering that $a$ is EP if and only if $a^{\ast}$ is EP and having in mind
Theorem \ref{equations-a}, Theorem \ref{biao-b}, Theorem \ref{equations-b} and Theorem \ref{biao-c}, we get the following four theorems.

\begin{theorem} \label{equations-c}
Let $a\in R$. Then $a\in R^{\mathrm{EP}}$ if and only if there exists $y\in R$ such that
$$(ay)^{\ast}=ay,~~a^{2}y=a~~\text{and}~~y^{2}a=y.$$
\end{theorem}

\begin{theorem} \label{biao-bb}
Let $a \in R$. If
$a$ is EP, then $\{ y \in R: (ay)^* = ay, a^2y =a, y^2a = y \} = \{
a^\dag +(1-aa^\dag)xaa^\dag : x \in R\}$.
Moreover, if $R$ is a prime ring, then
$\{ y \in R: (ay)^* = ay, a^2y =a, y^2a = y \} = \{ a^\dag\}$ if and only if $a=0$ or $a$ is invertible.
\end{theorem}

\begin{theorem} \label{equations-d}
Let $a\in R$. Then the following are equivalent:
\begin{itemize}
\item[{\rm (1)}] $a\in R^{\mathrm{EP}}$;
\item[{\rm (2)}] there exists $y\in R$ such that $aya=a$, $Ry=Ra$ and $y^{\ast}R=aR$;
\item[{\rm (3)}] there exists $y\in R$ such that $aya=a$, $Ry=Ra$ and $y^{\ast}R\subseteq aR$;
\item[{\rm (4)}] there exists $y\in R$ such that $yay=y$, $Ry=Ra$ and $y^{\ast}R=aR$;
\item[{\rm (5)}] there exists $y\in R$ such that $yay=y$, $Ry=Ra$ and $aR\subseteq y^{\ast}R$;
\item[{\rm (6)}] there exists $y\in R$ such that $aya=a$, $y^{\circ}=a^{\circ}$
and ${}^{\circ}(y^{\ast})={}^{\circ}a$;
\item[{\rm (7)}] there exists $y\in R$ such that $aya=a$, $y^{\circ}=a^{\circ}$
and $^{\circ}a\subseteq {}^{\circ}(y^{\ast})$;
\item[{\rm (8)}] there exists $y\in R$ such that $yay=y$, $y^{\circ}=a^{\circ}$
and $^{\circ}(y^{\ast}) = {}^{\circ}a$;
\item[{\rm (9)}] there exists $y\in R$ such that $yay=y$, $y^{\circ}=a^{\circ}$
and $^{\circ}(y^{\ast})\subseteq {}^{\circ}a$.
\end{itemize}
\end{theorem}

\begin{theorem} \label{biao-cc}
Let $a\in R^{EP}$ and denote $p=aa^{\dagger}$. Then the following sets are the same
\begin{itemize}
\item[{\rm (1)}] $\{y\in R: aya=a, Ry \subseteq Ra\}=\{a^{\dagger}+(1-p)xp:x\in R\}$;
\item[{\rm (2)}] $\{y\in R: yay=y, Ry=Ra\}=\{a^{\dagger}+(1-p)xp:x\in R\}$;
\item[{\rm (3)}] $\{y\in R: aya=a, a^{\circ}\subseteq y^{\circ}\} =
\{a^{\dagger}+(1-p)xp : x\in R\}$;
\item[{\rm (4)}] $\{y\in R: yay=y, a^{\circ} = y^{\circ}\} =
\{a^{\dagger}+(1-p)xp : x\in R\}$.
\end{itemize}
Furthermore, if $R$ is prime, then any of
the above subsets is a singleton if and only if $a=0$ or $a$ is invertible.
\end{theorem}

We will characterize when $a \in R$ is EP by another three equations.

\begin{theorem} \label{julio-a}
Let $a\in R$. Then $a\in R^{\mathrm{EP}}$ if and only if there exists $x\in R$ such that
\begin{eqnarray}\label{j-a}
a^{2}x=a, ~~ ax=xa ~~ \text{and} ~~ (ax)^{\ast}=ax.
\end{eqnarray}
\end{theorem}
\begin{proof}
If $a$ is EP, by taking $x=a^{\dagger}=a^\#$, we get (\ref{j-a}).
Conversely, assume that exists $x \in R$ such that (\ref{j-a}) is satisfied.
We shall show that $a\in R^{\mathrm{EP}}$ and $a^\#=ax^{2}.$
Since $ax=xa$, we get $a(ax^{2})=(ax^{2})a$, but in addition, $a(ax^{2})=(a^{2}x)x=ax$,
which leads to $a(ax^{2})a=a^{2}x=a$ and
$(ax^{2})a(ax^{2})=(ax^{2})ax=(a^{2}x)x^{2}=ax^{2}.$
Since $aa^\#=a^{2}x^{2}=ax$ is Hermitian, the conclusion follows from Lemma \ref{group-aa-aab}.
\end{proof}

We have seen that
if $a\in R$ is EP, then $\{ x \in R: a^2x=a, ax=xa, (ax)^*=ax \}$ is not empty.
In next theorem we describe this set.

\begin{theorem} \label{julio-biao-a}
Let $a \in R$. If
$a$ is EP, then $\{ x \in R: a^{2}x=a, ax=xa, (ax)^{\ast}=ax \} = \{
a^\dag +(1-aa^\dag) y (1-aa^\dag): y \in R\}$.
Moreover, if $R$ is a semiprime ring, then
$\{ x \in R: a^{2}x=a, ax=xa, (ax)^{\ast}=ax \} = \{ a^\dag\}$ if and only if $a$ is invertible.
\end{theorem}
\begin{proof}
Suppose that $a$ is an EP element.
We use the matrix representations
of $a$ and $a^\dag$ with respect to the projection $p=aa^\dagger$ given in (\ref{mra}).
Let $x=\left[\begin{smallmatrix} u & v \\ w & z \end{smallmatrix} \right]$ be the representation
of any $x\in R$ with respect to $p$.

Let $x \in R$ satisfy $a^2x=a$, $ax=xa$ and $(ax)^*=ax$.
From $a^{2}x=a$, we get
$$
\left[\begin{matrix} a & 0 \\ 0 & 0 \end{matrix} \right]^{2}
\left[\begin{matrix} u & v \\ w & z \end{matrix} \right]
=\left[\begin{matrix} a & 0 \\ 0 & 0 \end{matrix} \right],
$$
which leads to $a^2u=a$ and $a^2 v=0$.
Since $a$ is EP and $aa^{\dagger}=p=a^{\dagger}a$, then $a$ is invertible in $pRp$ and its inverse is $a^{\dagger}$.
Hence from $a=a^{2}u$ and $0=a^{2}v$, we obtain $u=a^{\dagger}$ and $0=v$, respectively. Now, from $ax=xa$ we have
$$
\left[\begin{matrix} a & 0 \\ 0 & 0 \end{matrix} \right]
\left[\begin{matrix} a^{\dagger} & 0 \\ w & z \end{matrix} \right]
= \left[\begin{matrix} a^{\dagger} & 0 \\ w & z \end{matrix} \right]
\left[\begin{matrix} a & 0 \\ 0 & 0 \end{matrix} \right],
$$
which implies $0=wa$, and taking into account that $a$ is invertible in
$pRp$, we have $0=w$. Therefore,
$$
x=\left[\begin{matrix} a^{\dagger} & 0 \\ 0 & z \end{matrix} \right]
=a^{\dagger}+z,
$$
that is $\{ x \in R: a^{2}x=a, ax=xa, (ax)^{\ast}=ax \}\subseteq\{a^{\dagger}+(1-aa^{\dagger})y(1-aa^{\dagger}):y\in R\}.$

Let us prove the opposite inclusion.
We have $aa^{\dagger}=a^{\dagger}a$ since $a$ is EP. Now,
\begin{equation*}
\begin{split}
~
&a[a^{\dagger}+(1-aa^{\dagger})y(1-aa^{\dagger})]=aa^{\dagger} \text{~is Hermitian},\\
&a^{2}[a^{\dagger}+(1-aa^{\dagger})y(1-aa^{\dagger})]=a^{2}a^{\dagger}=a,\\
&a[a^{\dagger}+(1-aa^{\dagger})y(1-aa^{\dagger})]
=aa^{\dagger}=a^{\dagger}a
=[a^{\dagger}+(1-aa^{\dagger})y(1-aa^{\dagger})]a.
\end{split}
\end{equation*}

Suppose that $R$ is a semiprime ring.
If $a$ is invertible, then $\{ x \in R: a^{2}x=a, ax=xa, (ax)^{\ast}=ax \} = \{a^{-1}\}$.
If $\{ x \in R: a^{2}x=a, ax=xa, (ax)^{\ast}=ax \}$ is a singleton, then $(1-aa^\dag)y(1-aa^\dag)=0$ for all $y \in R$. By
using that $R$ is semiprime, we get $1-aa^\dagger=0$, which (since $a$ is EP)
is equivalent to the invertibility of $a$.
\end{proof}

\begin{theorem} \label{julio-b}
Let $a\in R$. Then the following are equivalent:
\begin{itemize}
\item[{\rm (1)}] $a\in R^{\mathrm{EP}}$;
\item[{\rm (2)}] $^{\circ}(a^{2})\subseteq {}^{\circ}a$ and there exists $x\in R$ such that $xa^{2}=a$ and $(xa)^{\ast}=xa$;
\item[{\rm (3)}] $(a^{2})^{\circ}\subseteq a^{\circ}$ and there exists $x\in R$ such that $a^{2}x=a$ and $(ax)^{\ast}=ax$.
\end{itemize}
Furthermore, under these equivalences one has that the set of elements $x$ satisfying $(2)$ is
$\{a^{\dagger}+y(1-aa^{\dagger}): y\in R\}$
and the set of elements $x$ satisfying $(3)$ is
$\{a^{\dagger}+(1-aa^{\dagger})z: z\in R\}$.
\end{theorem}
\begin{proof}
(1) $\Rightarrow$ (2): The inclusion $^{\circ}(a^{2})\subseteq {}^{\circ}a$
is evident from $a\in R^\#$. For the remaining,  it is sufficient to take $x=a^\dagger=a^\#$.

(2) $\Rightarrow$ (1): Since $(ax-1)a^2 = a(xa^2)-a^2 = a^2-a^2=0$, we get
$ax-1 \in$ $^{\circ}(a^{2})\subseteq {}^{\circ}a$, hence $axa=a$. From
$$ax^2a^2 = ax(xa^2)=axa=a=xa^2$$
we get $ax^2-x \in {}^{\circ}(a^{2})\subseteq {}^{\circ}a$, hence $ax^2a=xa$. Now, we prove
$a^\# = x^2a$ by the definition of the group inverse,
$$
a(x^2a) = ax^2a = xa, \qquad (x^2a)a = x(xa^2)=xa;
$$
$$
a(x^2a)a = xa^2 = a; \qquad (x^2a)a(x^2a) = x(xa^2)x^2a = xax^2a=x(ax^2a)=x^2a.
$$
Now, $a^\# = x^2a$ and $aa^\# =ax^2a=xa$ is Hermitian. Hence (1) follows from
Lemma~\ref{group-aa-aab}.

The proof of $(1) \Leftrightarrow (3)$ is similar to the proof
of $(1)\Leftrightarrow(2)$.



Now, let us prove the last part of the theorem.
Recall that $a$ is EP.
If $x \in R$ satisfies $xa^2=a$, then $(x-a^\dag)aa^\dag = xaa^\dag -a^\dag =
xa^2(a^\dag)^2-a^\dag = a(a^\dag)^2-a^\dag = 0$. Hence
$x-a^\dag = (x-a^\dag)(1-aa^\dag)$, which yields
$x \in \{ a^\dag + y(1-aa^\dag): y \in R\}$. Reciprocally,
it is evident that for any $y \in R$ one has that
$[a^\dag + y(1-aa^\dag)]a^2 = a$ and $[a^\dag + y(1-aa^\dag)]a$ is Hermitian.
\end{proof}

\section { \bf  When a core invertible element is an EP element}

Any EP element is core invertible, but when a core invertible element is EP?
In this section we answer this question. Let us start this section with a lemma.

\begin{lemma} \emph{\cite[Theorem 2.14]{RDD}} \label{five-equations}
An element $a \in R$ is core invertible if and only if there exists
$x \in R$ such that
$$axa=a,~~xax=x,~~(ax)^{\ast}=ax,~~xa^{2}=a~~\text{and}~~ax^{2}=x.$$
Under this equivalence, one has that $x=\core{a}$.
\end{lemma}

Let us recall the following result.
\begin{theorem}\label{t32}
Let $a \in R$.
\begin{itemize}
\item[{\rm (1)}] {\rm \cite[Proposition 8.24]{BR}} $a$ is group invertible if and only if
there exists an idempotent $p \in R$ such that $ap=pa=0$ and $a+p$ is invertible.
Under this equivalence, we have $p=1-aa^\#$ and $a^\# = (a+p)^{-1}-p$.
\item[{\rm (2)}] {\rm \cite[Theorem 2.1]{BJ}} $a$ is EP if and only if there exists a
projection $p \in R$ such that $ap=pa=0$ and $a+p$ is invertible.
Under this equivalence, we have $p=1-aa^\dag$ and $a^\dag = (a+p)^{-1}-p$.
\end{itemize}
\end{theorem}
In fact, the second item of previous result was stated for unital $C^*$-algebras, but
as one can easily check, its proof remains valid for unital rings with an involution.
We give a similar characterization of the core invertibility.

\begin{theorem}\label{t33}
Let $a \in R$. Then following are equivalent:
\begin{itemize}
\item[{\rm (1)}] $a$ is core invertible;
\item[{\rm (2)}] exists a projection $p$ such that $pa=0$ and $a(1-p)$ is invertible in the ring
$(1-p)R(1-p)$;
\item[{\rm (3)}] exists a projection $p$ such that $pa=0$ and $a(1-p)+p$ is invertible.
\end{itemize}
Under this equivalence, one has that this projection $p$ is unique and $p=1-a \core{a}$. In addition,
$$
(a(1-p))^{-1}_{(1-p)R(1-p)} = \core{a}, \quad
(a(1-p)+p)^{-1} = p+\core{a}.
$$
\end{theorem}
\begin{proof}
(1) $\Rightarrow$ (2): Let $p=1-a\core{a}$, then $p$ is a projection. We use the notation $\pn=1-p=a\core{a}$.
Observe that from Lemma \ref{five-equations},
we have $a \pn =(1-p)a \pn =\pn a \pn\in \pn R \pn$ by $pa=0,$
and $\core{a}=a\core{a}\core{a}=a\core{a}\core{a}a\core{a}\in \pn R \pn.$
From $\core{a}a^2=a$ we have $\core{a}a^2 \core{a} = a \core{a}$, i.e.,
$\core{a} a \pn = \pn$. Furthermore, $a\pn \core{a} = a \core{a} = \pn$.
Therefore, $a \pn \in \pn R \pn$ is invertible in the ring
$\pn R \pn$ and its inverse is $\core{a}$.

(2) $\Leftrightarrow$ (3): 
Let $p \in R$ be a projection such that $pa=0$. It is easy to verify that
\begin{equation}\label{proof_33_matrix}
a(1-p) = \left[ \begin{matrix} 0 & 0 \\ 0 & a (1-p) \end{matrix} \right], \quad
a(1-p) + p = \left[ \begin{matrix} p & 0 \\ 0 & a (1-p) \end{matrix} \right].
\end{equation}
Taking into account that $p$ is invertible in the ring $pRp$ (in fact, $p$ is the unity),
evidently we have that $a (1-p) \in ((1-p)R (1-p))^{-1} \Leftrightarrow a(1-p) + p \in R^{-1}$.

(2) $\Rightarrow$ (1): Let $x \in (1-p) R (1-p)$ be the inverse of $a (1-p)$ in $(1-p) R (1-p)$.
This means that $a(1-p)x = xa(1-p) = 1-p$.
Observe that $x \in (1-p) R (1-p)$ implies $(1-p)x=x(1-p)=x$, and therefore,
$ax = a(1-p)x = xa(1-p) = 1-p$. We will prove that $x=\core{a}$ by using Lemma \ref{five-equations}.
Let us recall that we can use $pa=0$ and $(1-p)a=a$ by hypothesis.
$$
axa = (ax)a = (1-p)a = a.
$$
$$
xax = x(ax) = x(1-p) = x.
$$
$$
ax = 1-p \text{ is Hermitian}.
$$
$$
xa^2 = xa(1-p)a = (1-p)a = a.
$$
$$
ax^2 = (ax)x = (1-p)x = x.
$$

Now, we shall prove the uniqueness of the projection $p$. Assume that $q$ is another projection
such that $qa = 0$ and $a(1-q)$ is invertible in $(1-q)R(1-q)$. By the proof of
(2) $\Rightarrow$ (1) we get that the inverse of $a(1-q)$ in $(1-q)R(1-q)$ is
$\core{a}$, in particular $\core{a} \in (1-q) R (1-q)$ and $a(1-q)\core{a} = 1-q$.
By using also Lemma \ref{five-equations} we get
\begin{equation}\label{proofth33}
(1-q) a \core{a} = a(1-q)\core{a} a \core{a} = a(1-q) \core{a} = 1-q.
\end{equation}
Since $\core{a} \in (1-q)R(1-q)$, exists $u \in R$ such that $\core{a}=u(1-q)$. Now,
$$
a \core{a} (1-q) = au(1-q)^2 = au(1-q)=a\core{a}.
$$
Apply involution and use Lemma \ref{five-equations} in this last equality to get
$(1-q) a\core{a} = a\core{a}$. From this last equality and (\ref{proofth33}) we obtain
$a \core{a} = 1-q$. In other words, we have proved the uniqueness of such $q$.

Finally, $(a(1-p))^{-1}_{(1-p)R(1-p)} = \core{a}$
holds in view of the proof of (2) $\Leftrightarrow$ (1).
Let us represent $a$ and $\core{a}$ with respect to the
projection $p=1-a\core{a}$. From the matrix representation of $a(1-p)+p$ given in
(\ref{proof_33_matrix}) we easily get $(a(1-p)+p)^{-1} = p+
(a(1-p))^{-1}_{(1-p)R(1-p)} = p+\core{a}$. 
The proof is finished.
\end{proof}

\begin{example} \label{example111}
\emph{ If $a\in R$ is core invertible element and there exists a projection $p\in R$
such that $ap=0$, we could not get that $a\in R^{\mathrm{EP}}$. Consider the
following counterexample. Let $R$ be the ring of all complex $2\times2$ matrices. Let
$$
A=\left[\begin{matrix} 0 & 1 \\ 0 & 1 \end{matrix} \right] \quad \text{and} \quad
P=\left[\begin{matrix} 1 & 0 \\ 0 & 0 \end{matrix} \right].
$$
It is simple to prove that
$$
A^\#=\left[\begin{matrix} 0 & 1 \\ 0 & 1 \end{matrix} \right], \quad
\core{A}=\left[\begin{matrix} 1/2 & 1/2 \\ 1/2 & 1/2 \end{matrix}\right], \quad
AP=\left[\begin{matrix} 0 & 0 \\ 0 & 0 \end{matrix} \right],
$$
and $P^{2}=P=P^{\ast}$. But $A^\# \neq \core{A}$, therefore $A$ is not EP.
The projection given by Theorem~\ref{t33} is $I-A\core{A} =
\left[ \begin{smallmatrix} 1/2 & -1/2 \\ -1/2 & 1/2 \end{smallmatrix} \right]$.
This projection satisfies $(I-A\core{A})A=0$ and $A(I-A\core{A}) \neq 0$.
}
\end{example}


The following lemmas are necessary to prove next Theorem \ref{ep-xuyang}.
We will use the notation $[a,b]=ab-ba$ in the next lemma.

\begin{lemma} \emph{\cite[Theorem 3.1]{RDD}} \label{core-epa}
Let $a\in R$. Then the following are equivalent:
\begin{itemize}
\item[{\rm (1)}] $a\in R^{\mathrm{EP}}$;
\item[{\rm (2)}] $a\in R^{\dagger}$ and $[a,a^{\dagger}]=0$;
\item[{\rm (3)}] $a\in \core{R}$ and $[a,\core{a}]=0$;
\item[{\rm (4)}] $a\in \core{R}$ and $a^\#=\core{a}$;
\item[{\rm (5)}] $a\in R^{\dagger}\cap R^\#$ and $a^{\dagger}=\core{a}$.
\end{itemize}
\end{lemma}

\begin{lemma} \label{cxzzj}
Let $a\in \core{R}$. Then
\begin{itemize}
\item[{\rm (1)}] \emph{\cite[Theorem 2.19]{RDD}}
$a\in R^\#$ and $a^\#=(\core{a})^{2}a$;
\item[{\rm (2)}] \emph{\cite[Theorem 2.18]{RDD}}
$\core{a} \in R^{\mathrm{EP}}$ and $\core{(\core{a})} =
(\core{a})^{\dagger} = (\core{a})^\# = a^2 \core{a}$. Moreover, if $a\in R^{\dagger}$,
then $\core{(a^\dag)} = (\core{a}a)^{\ast}a$.
\end{itemize}
\end{lemma}

\begin{theorem} \label{ep-xuyang}
Let $a\in R$. Then the following are equivalent:
\begin{itemize}
\item[{\rm (1)}] $a \in R^{\mathrm{EP}}$;
\item[{\rm (2)}] $a \in \core{R}$ and $(\core{a}a)^{\ast} = \core{a}a$;
\item[{\rm (3)}] $a \in \core{R}$ and $\core{(\core{a})} = a$;
\item[{\rm (4)}] $a \in \core{R}$ and $(\core{a})^{\dagger}=a$;
\item[{\rm (5)}] $a \in \core{R}$ and $(\core{a})^\#=a$;
\item[{\rm (6)}] $a\in R^{\dagger}\cap R^\#$ and $\core{(a^{\dagger})} = a$;
\item[{\rm (7)}] $a\in R^{\dagger}\cap R^\#$ and $\core{(a^{\dagger})} = (\core{a})^{\dagger}$;
\item[{\rm (8)}] $a\in \core{R}$ and $ap=0$, where $p = 1-a\core{a}$.
\end{itemize}
\end{theorem}
\begin{proof}
$(1)\Leftrightarrow(2)$: Suppose $a\in R^{\mathrm{EP}}$. Then by Lemma \ref{core-epa},
we have $a^{\dagger} = \core{a}$. Thus $(a^{\dagger}a)^{\ast}=a^{\dagger}a$ implies
$(\core{a}a)^{\ast} = \core{a}a$. Conversely, suppose $a\in \core{R}$ and $(\core{a}a)^{\ast}
= \core{a}a$. By Lemma \ref{five-equations}, we have $a \core{a} a = a$, $\core{a} a \core{a}
= \core{a}$ and $(a\core{a})^{\ast} = a\core{a}$. Thus by the definition of Moore-Penrose
inverse, we have $a^{\dagger} = \core{a}$. Hence by Lemma~\ref{core-epa}, we have
$a\in R^{\mathrm{EP}}$.

$(1)\Leftrightarrow(3)$: Suppose $a\in R^{\mathrm{EP}}$. Then by Lemma~\ref{core-epa}, we have
$[a,\core{a}]=0$. By Lemma~\ref{cxzzj}, we have $\core{(\core{a})} =  a^2 \core{a}$. Thus
$\core{(\core{a})} = a^2 \core{a} = a(a\core{a}) = a(\core{a}a)=a$. Conversely, suppose
$\core{(\core{a})} = a$. By Lemma~\ref{cxzzj}, we have $a  = \core{(\core{a})} = a^2 \core{a}$.
Thus
$$
\core{a}a = \core{a} a^2 \core{a} = a\core{a}.
$$
Therefore $a\in R^{\mathrm{EP}}$ by Lemma \ref{core-epa}.

$(3)\Leftrightarrow(4)\Leftrightarrow(5)$ is clear by Lemma \ref{cxzzj}.

$(1)\Rightarrow(6)$: By Lemma \ref{core-epa} and Lemma \ref{cxzzj}, we have
$\core{(a^\dagger)} = (\core{a}a)^{\ast}a =(a^{\dagger}a)^{\ast}a=a^{\dagger}a^{2}=a^\#a^2=a$.

$(6)\Rightarrow(7)$: Suppose that $a\in R^{\tiny\textcircled{\tiny\#}}$ and $(a^{\dagger})^{\tiny\textcircled{\tiny\#}}=a$. Then by Lemma~ \ref{five-equations},
we have $a^{\dagger}a^{2}=a$ and $a(a^{\dagger})^{2}=a^{\dagger}$.
Since $a^\dag a$ is Hermitian, an appeal to Theorem~\ref{equations-a}
leads to $a \in R^{\rm EP}$. Hence $\core{a}=a^\#$. By Lemma \ref{cxzzj}, we have
$(\core{a})^{\dagger}=a^2 \core{a}=a^2 a^\#=a$. Thus by $\core{(a^{\dagger})}=a$, we have
$\core{(a^{\dagger})}=(\core{a})^{\dagger}$.

$(7)\Rightarrow (1)$: Suppose that $a\in \core{R}$ and $\core{(a^{\dagger})} =
(\core{a})^{\dagger}$. By Lemma \ref{cxzzj}, we have $(\core{a})^{\dagger} = a^2 \core{a}$ and
$\core{(a^{\dagger})} = (\core{a}a)^{\ast}a$. Thus by  $\core{(a^{\dagger})} =
(\core{a})^{\dagger}$, we have
\begin{eqnarray}\label{xu-yang3}
a^{2} \core{a} = (\core{a} a)^{\ast}a.
\end{eqnarray}
Taking involution on (\ref{xu-yang3}), we have
$a^{\ast} \core{a}a = (a\core{a})^{\ast}a^{\ast} = a\core{a}a^{\ast}$.
Thus by Lemma~\ref{five-equations}, we have $\core{a}a = a (\core{a})^2 a = (a \core{a})\core{a}a =
(a \core{a})^* \core{a}a =  (\core{a})^{\ast}a^{\ast}\core{a}a
= (\core{a})^* a \core{a} a^* = (\core{a})^* (a \core{a})^* a^* =
(a^2 (\core{a})^2)^* = (a \core{a})^* = a\core{a}$.
That is,  $[a,\core{a}]=0$, therefore $a\in R^{\mathrm{EP}}$ by Lemma \ref{core-epa}.

$(8)\Leftrightarrow(1)$ follows
by Theorem \ref{t32} and Theorem \ref{t33}.
\end{proof}

In the following Theorem \ref{ep-core-baohan}, we show that the equality
$aR=a^{\ast}R$ in Lemma \ref{mp-sharpEPinvo} can be
replaced by the weaker inclusions $aR\subseteq a^{\ast}R$ or $a^{\ast}R\subseteq aR$.

\begin{lemma}  \emph{\cite[Proposition 2]{PP}} \label{mp-sharpEPinvo}
Let $a\in R$. Then the following are equivalent:
\begin{itemize}
\item[{\rm (1)}] $a\in R^{\mathrm{EP}}$;
\item[{\rm (2)}] $a\in R^\#$ and $aR=a^{\ast}R$;
\item[{\rm (3)}] $a\in R^\#$ and $Ra=Ra^{\ast}$;
\item[{\rm (4)}] $a\in R^{\dagger}$ and $aR=a^{\ast}R$;
\item[{\rm (5)}] $a\in R^{\dagger}$ and $Ra=Ra^{\ast}$.
\end{itemize}
\end{lemma}


\begin{theorem}  \label{ep-core-baohan}
Let $a\in R$. Then the following are equivalent:
\begin{itemize}
\item[{\rm (1)}] $a\in R^{\mathrm{EP}}$;
\item[{\rm (2)}] $a\in R^\#$ and $aR\subseteq a^{\ast}R$;
\item[{\rm (3)}] $a\in R^\#$ and $Ra\subseteq Ra^{\ast}$;
\item[{\rm (4)}] $a\in R^\#$ and $a^{\ast}R\subseteq aR$;
\item[{\rm (5)}] $a\in R^\#$ and $Ra^{\ast}\subseteq Ra$.
\end{itemize}
\end{theorem}
\begin{proof}
$(1)\Rightarrow(2)$--$(5)$ is obvious by Lemma \ref{mp-sharpEPinvo}.

$(2)\Rightarrow(1)$: By $aR\subseteq a^{\ast}R$, we have $a=a^{\ast}r$ for some $r\in R$, then
$a=(aa^\#a)^{\ast}r=(a^\#a)^{\ast}a^{\ast}r=(a^\#a)^{\ast}a$.
Thus $a^\#a=aa^\#=(a^\#a)^{\ast}aa^\#=(a^\#a)^{\ast}a^\#a$, which gives $(a^\#a)^{\ast}=a^\#a.$
Therefore $a\in R^{\rm EP}$ by the definition of EP element.

$(3)$--$(5)\Rightarrow(1)$ is similar to $(2)\Rightarrow(1)$.
\end{proof}

\begin{corollary}  \label{ep-core-baohan-cor}
Let $a\in R$. Then the following are equivalent:
\begin{itemize}
\item[{\rm (1)}] $a\in R^{\mathrm{EP}}$;
\item[{\rm (2)}] $a\in \core{R}$ and $aR\subseteq a^{\ast}R$;
\item[{\rm (3)}] $a\in \core{R}$ and $Ra\subseteq Ra^{\ast}$;
\item[{\rm (4)}] $a\in \core{R}$ and $a^{\ast}R\subseteq aR$;
\item[{\rm (5)}] $a\in \core{R}$ and $Ra^{\ast}\subseteq Ra$.
\end{itemize}
\end{corollary}
\begin{proof}
It is obvious by Lemma \ref{cxzzj} and Theorem \ref{ep-core-baohan}.
\end{proof}




\begin{theorem}  \label{ep-core-1}
Let $a\in \core{R}$. Then $a\in R^{\mathrm{EP}}$ if and only if $[\core{a},(\core{a}a)^{\ast}a]=0$.
\end{theorem}
\begin{proof}
If $a$ is EP, then $a\in \core{R}$ and $a^\# = a^{\dagger} = \core{a}$.
Thus $[\core{a},(\core{a}a)^{\ast}a]
=[a^\#,(a^{\dagger}a)^{\ast}a]=[a^\#,a^{\dagger}a^{2}]=[a^\#,a^\#a^{2}]=[a^\#,a]=0.$

Conversely, if $[\core{a},(\core{a}a)^{\ast}a]=0$, then
\begin{eqnarray}\label{mpeq7}
\core{a}(\core{a}a)^{\ast}a = (\core{a}a)^{\ast}a\core{a}.
\end{eqnarray}
Taking involution $\ast$ on (\ref{mpeq7}), in view of Lemma~\ref{five-equations}, we get
$a^* \core{a}a(\core{a})^* = a (\core{a})^2 a = \core{a} a$,
Thus, $a^\ast \core{a} a(\core{a})^\ast a = \core{a}a^2=a$. Hence, $aR\subseteq a^{\ast}R$.
Therefore, $a\in R^{\mathrm{EP}}$ by Corollary \ref{ep-core-baohan-cor}.
\end{proof}

In \cite[Theorem 16]{BO}, for an operator $T\in L(X)$, where $X$ is a Banach space, Boasso proved that
for a Moore-Penrose invertible operator $T$,
$T$ is an EP operator if and only if there exists an invertible operator
$P\in L(X)$ such that $T^{\dagger}=PT.$
Inspired by this result, we get the following theorem.

\begin{theorem} \label{core-unita}
Let $a\in R^{\tiny\textcircled{\tiny\#}}$. Then the following are equivalent:
\begin{itemize}
\item[{\rm (1)}] $a\in R^{\mathrm{EP}}$;
\item[{\rm (2)}] there exists a unit $u\in R$ such that $\core{a} = ua$;
\item[{\rm (3)}] there exists an element $b\in R$ such that $\core{a}=ba$.
\end{itemize}
\end{theorem}
\begin{proof}
$(1)\Rightarrow(2)$: If $a\in R^{\mathrm{EP}}$, then $a\in \core{R}$ and $\core{a}=a^\#$.
Let $u=(a^\#)^{2}+1-aa^\#$. Since
$u(a^{2}+1-aa^\#)=(a^{2}+1-aa^\#)u=1$, we get that $u$ is a unit.
Furthermore, we have $ua=((a^\#)^{2}+1-aa^\#)a=a^\#=a^{\tiny\textcircled{\tiny\#}}$.

$(2)\Rightarrow(3)$ is clear.

$(3)\Rightarrow(1)$:
We know that $R \core{a} = Ra^*$ by the definition of the core inverse. From $\core{a} = ba$ we
get $R \core{a} \subseteq Ra$. Thus $Ra^* = Ra^{\tiny\textcircled{\tiny\#}} \subseteq Ra$.
Therefore, we deduce that $a\in R^{\mathrm{EP}}$ by Corollary \ref{ep-core-baohan-cor}.
\end{proof}

By the previous result, we know that if $a\in R$ is an EP element, then
the equation $\core{a} = xa$ has at least one solution. In fact, one solution
is $x=(\core{a})^2+1-a\core{a}$, as one can see in the proof
of (1) $\Rightarrow$ (2) of the previous result. We will
describe the set of solutions in next result.

\begin{theorem}
If $a \in R^{\rm EP}$, then $\{ x \in R: \core{a} = xa \} = \{ (\core{a})^2+w(1-a\core{a}):
w \in R \}$.
\end{theorem}
\begin{proof}
Let $x\in R$ such that $\core{a}=xa$. Since $x = x-xa\core{a}+(\core{a})^2 =
x(1-a\core{a})+(\core{a})^2$ we have proved the ``$\subseteq$'' inclusion of the statement
of the theorem.

Let us prove the opposite inclusion. Since $a \in R^{\rm EP}$, by Theorem~\ref{core-unita}
there exists $b \in R$ such that $\core{a}=ba$. Now, $(\core{a})^2 a = ba \core{a} a=ba=\core{a}$.
Finally, if $w$ is any element of $R$, then $\left[ (\core{a})^2 +w (1-a\core{a})\right]a  =
\core{a}$.
\end{proof}

\section { \bf  When a Moore-Penrose invertible element is an EP element}

Since any EP element is Moore-Penrose invertible, it is natural to ask when a Moore-Penrose invertible element is an EP element.
The concept of bi-EP was introduced by Hartwig and Spindelb\"{o}ck in \cite{HS} for complex matrices.
They proved that for a complex $A\in \mathbb{C}_{_{n\times n}}$, if $A$ is group invertible, then
$A$ is an EP matrix if and only if $A$ is bi-EP. We give a generalization of this result in Theorem \ref{g-bi-ep-thm}.
In this section, we give the definition of $n$-EP, which is a generalization of bi-EP.
We show that any $n$-EP element is an EP element whenever this element is group invertible.

\begin{definition} \emph{\cite{HS}}\label{bi-ep-defn}
An element $a\in R$ is called bi-EP if $a\in R^{\dagger}$ and $[aa^{\dagger},a^{\dagger}a]=0.$
\end{definition}


\begin{definition}  \label{generalized-bi-ep}
Let $n$ be a positive integer.
An element $a\in R$ is called $n$-EP if $a\in R^{\dagger}$ and $[a^na^{\dagger},a^{\dagger}a^n]=0.$
\end{definition}

Note that 1-EP is coincide with bi-EP.

In \cite[Theorem 2.1]{MD}, Mosi\'{c} and Djordjevi\'{c} proved that $a\in R^{\mathrm{EP}}$ if and only if $a\in R^\#\cap R^{\dagger}$ and
$a^{n}a^{\dagger}=a^{\dagger}a^{n}$ for some $n\geqslant 1$. This result also can be found in \cite[Theorem~2.4]{C} by Chen.
In the following theorem, we give a generalization of this result.

\begin{theorem} \label{g-bi-ep-thm}
Let $a\in R$ and $n$ be a positive integer. Then $a\in R^{\mathrm{EP}}$ if and only if $a\in R^{\dagger}\cap R^\#$ and a is $n$-EP.
\end{theorem}
\begin{proof}
Suppose $a\in R^{\mathrm{EP}}$. Then $[a,a^{\dagger}]=0$,
which gives $[a^{n}a^{\dagger},a^{\dagger}a^{n}]=0.$ That is, $a$ is $n$-EP.

Conversely, suppose that $a\in R^{\dagger}\cap R^\#$ and $a$ is $n$-EP. Then we have
\begin{eqnarray}\label{biepeqg1}
a^{\dagger}a^{2n}a^{\dagger}=a^{n}(a^{\dagger})^{2}a^{n}.
\end{eqnarray}
Pre-multiplication and post-multiplication of (\ref{biepeqg1}) by $a$ respectively now yields
$a^{2n}a^{\dagger}=a^{n+1}(a^{\dagger})^{2}a^{n},$ and $a^{\dagger}a^{2n}=a^{n}(a^{\dagger})^{2}a^{n+1}.$
Thus
\begin{eqnarray}\label{biepeqg4}
a^{2n-1}a^{\dagger}=a^\#a^{2n}a^{\dagger}=a^\#a^{n+1}(a^{\dagger})^{2}a^{n}=a^{n}(a^{\dagger})^{2}a^{n}.\\
\label{biepeqg5}
a^{\dagger}a^{2n-1}=a^{\dagger}a^{2n}a^\#=a^{n}(a^{\dagger})^{2}a^{n+1}a^\#=a^{n}(a^{\dagger})^{2}a^{n}.
\end{eqnarray}
By (\ref{biepeqg4}) and (\ref{biepeqg5}), we have $a^{2n-1}a^{\dagger}=a^{\dagger}a^{2n-1}.$
Hence by \cite[Theorem 2.1]{MD}, we have $a\in R^{\mathrm{EP}}.$
\end{proof}

\begin{theorem} \label{core-unit-bis}
Let $a\in R^{\dagger}$. Then the following are equivalent:
\begin{itemize}
\item[{\rm (1)}] $a\in R^{\mathrm{EP}}$;
\item[{\rm (2)}] there exists a unit $u\in R$ such that $a^{\dagger}=ua$;
\item[{\rm (3)}] there exists a left invertible element $v\in R$ such that $a^{\dagger}=va$.
\end{itemize}
\begin{proof}
$(1)\Rightarrow(2)$: If $a\in R^{\mathrm{EP}}$, then $a\in R^{\dagger}$ and $a^{\dagger}=a^\#.$
Let $u=(a^\#)^{2}+1-aa^\#$. Since
$u(a^{2}+1-aa^\#)=(a^{2}+1-aa^\#)u=1$ we get that $u$ is a unit. Furthermore,
$ua=((a^\#)^{2}+1-aa^\#)a=a^\#=a^{\dagger}$.

$(2)\Rightarrow(3)$ is clear.

$(3)\Rightarrow(1)$ Suppose that there exists a left invertible element $v\in R$ such
that $a^{\dagger}=va$. Then $1=tv$ for some $t\in R$ and $ta^{\dagger}=tva=a.$  Thus $Ra^{\dagger}\subseteq Ra$ and $Ra\subseteq Ra^{\dagger}$.
Since $Ra^{\dagger}=Ra^{\ast}$, we deduce that $Ra^{\ast} = Ra$, that is $a$ is an EP element.
\end{proof}
\end{theorem}

\begin{remark} \emph{In Theorem \ref{core-unita}, we proved that for a core invertible element $a\in R$, $a\in R^{\mathrm{EP}}$ if and only if
there exists an element $b\in R$ such that $a^{\tiny\textcircled{\tiny\#}}=ba$. The following example shows that this affirmation can not
be obtained for a Moore-Penrose invertible element. In Theorem \ref{ep-core-baohan}, we proved that for a group invertible element $a\in R$, $a\in R^{\mathrm{EP}}$ if and only if
$aR\subseteq a^{\ast}R$. The following example also shows that this affirmation can not
be obtained for a Moore-Penrose invertible element.
Observe that if $a \in R^\dag \cap R^\#$, then $a \in R^{\rm EP}$ if and only if
there exists $b \in R$ such that $a^\dag = ba$,
which follows from Theorem \ref{ep-core-baohan}.}
\end{remark}
Recall that an infinite matrix $M$ is said to be {\em bi-finite} if it is both row-finite and column-finite.

\begin{example} \label{example11}
\rm{
Let $R$ be the ring of all bi-finite real matrices with transpose as involution
and let $e_{i,j}$ be the matrix in $R$ with 1 in the $(i, j)$ position and 0 elsewhere.
Let
$A=\sum\limits_{i=1}^{\infty}e_{i+1,i}$ and $B=A^{\ast}$, now
$AB=\sum\limits_{i=2}^{\infty}e_{i,i}$,
$BA=I$. So $A^{\dagger}=B$ and $A^{\dagger}=A^{\dagger}BA=B^{2}A$.
It is easy to check that $B^{2}$ is not left invertible and $A$ is not EP (since $AB\neq BA$).
In addition, $A$ is not group invertible 
(if $A \in R^\#$, then $AA^\#=A^\#A=BAA^\#A=BA=I$. Note that $BA=I$, thus $A$ is invertible, which is not possible).
This example also shows that the equality $aR=a^{\ast}R$ in
the equivalence $(a \in R^{\rm EP} \Leftrightarrow a \in R^\dag, aR=a^{\ast}R)$
cannot be replaced by the inclusions $aR\subseteq a^{\ast}R$ or $a^{\ast}R\subseteq aR$.
}
\end{example}


\begin{proposition} \label{core-ep}
Let $a\in R^{\dagger}$. Then the following are equivalent:
\begin{enumerate}
\item[{\rm (1)}] $a\in R^{\mathrm{EP}}$;
\item[{\rm (2)}] $[a^{\dagger}a,a]=[a^{\dagger},aa^{\dagger}]=0$;
\item[{\rm (3)}] $[a^{\dagger}a,a]=[a,aa^{\dagger}]=0$;
\item[{\rm (4)}] $[a^{\dagger}a,a^{\dagger}]=[a^{\dagger},aa^{\dagger}]=0$;
\item[{\rm (5)}] $[a^{\dagger}a,a^{\dagger}]=[a,aa^{\dagger}]=0$.
\end{enumerate}
\end{proposition}
\begin{proof}
$(1)\Rightarrow(2)$--$(3)$: If $a\in R^{\mathrm{EP}}$, then $aa^{\dagger}=a^{\dagger}a$.
Thus, $(2)$ and $(3)$ are obvious.

$(2)\Rightarrow(1)$: Observe that $[a^{\dagger}a,a]=0$ implies that
$a=a^{\dagger}a^{2}\in a^{\ast}R$ and $[a^{\dagger},aa^{\dagger}]=0$ implies
that $a^{\dagger}=a(a^{\dagger})^{2}\in aR$, that is $a^{\ast}R\subseteq aR$ since $a^{\dagger}R=a^{\ast}R.$
Thus, $aR=a^{\ast}R$, i.e., $a$ is EP.

$(3)\Rightarrow(1)$:
Observe that $[a^{\dagger}a,a]=0$ implies that $a=a^{\dagger}a^{2}\in a^{\ast}R$ and
$[a,aa^{\dagger}]=0$ implies that $a=a^{2}a^{\dagger}\in Ra^{\dagger}$, that is
$Ra\subseteq Ra^*$ since $Ra^{\dagger}=Ra^{\ast}$, and therefore,
$a^{\ast}R\subseteq aR$. Thus, $aR=a^{\ast}R$, i.e., $a$ is EP.

The equivalence between $(1)$ $\Leftrightarrow$ $(4)$ $\Leftrightarrow$ $(5)$ is similar to the proof of the equivalence between
$(1)$ $\Leftrightarrow$ $(2)$ $\Leftrightarrow$ $(3)$.
\end{proof}

\begin{example} \label{example1}
\emph{ The condition $[a^{\dagger}a,a^{\dagger}]=0$ in Proposition \ref{core-ep} does not imply that
$a$ is an EP element in general.
Let $R$, $A$ and $B$ be the same as Example \ref{example11}, then
$AB=\sum\limits_{i=2}^{\infty}e_{i,i}$, $BA=I$. So $A^{\dagger}=B$
and $[A^{\dagger}A,A^{\dagger}]=0$. But $A$ is not EP since $AB\neq BA.$}
\end{example}



\begin{theorem} \label{ep-a=aa}
Let $a\in R^{\dagger}$. Then the following are equivalent:
\begin{itemize}
\item[{\rm (1)}] $a\in R^{\mathrm{EP}}$;
\item[{\rm (2)}] $aR=a^{2}R$ and $[a^{\dagger}a,a^{\dagger}]=0$;
\item[{\rm (3)}] $aR=a^{2}R$ and $[a^{\dagger}a,a]=0$;
\item[{\rm (4)}] $aR=a^{2}R$ and $aR\subseteq a^{\dagger}R$;
\item[{\rm (5)}] $aR=a^{2}R$ and $aR\subseteq a^{\ast}R$.
\end{itemize}
\end{theorem}
\begin{proof}
$(1)\Rightarrow$ $(2)$--$(5)$:
For a Moore-Penrose invertible element $a\in R$, we have $a\in R^{\rm EP}$ if and only if
$aa^{\dagger}=a^{\dagger}a$. Hence$(2)$--$(5)$ hold.

$(2)\Rightarrow(1)$: Observe that $[a^{\dagger}a,a^{\dagger}]=0$ implies $a^{\dagger}=(a^{\dagger})^{2}a$.
Thus $Ra^{\ast}\subseteq Ra$ (since $Ra^{\dagger}=Ra^{\ast}$). That is $a^{\ast}=ra$ for some $r\in R.$
We deduce that $a^{\ast}=ra=raa^{\dagger}a=a^{\ast}a^{\dagger}a$, applying involution on that last equality we obtain $a=a^{\dagger}a^{2}$.
Therefore $a\in R^{\mathrm{EP}}$ by Lemma \ref{annihilator} and Theorem \ref{julio-b}.

$(3)\Rightarrow(1)$: It is clear that $[a^{\dagger}a,a]=0$ implies $a=a^{\dagger}a^{2}$. Thus $a\in R^{\mathrm{EP}}$ by by Lemma \ref{annihilator} and Theorem \ref{julio-b}.

$(5)\Rightarrow(1)$: As $aR\subseteq a^{\ast}R$ is equivalent to $Ra^{\ast}\subseteq Ra$, we get $a\in R^{\mathrm{EP}}$ by the proof of $(2)\Rightarrow(1)$.

$(4)\Leftrightarrow(5)$: It is clear by $a^{\ast}R=a^{\dagger}R.$
\end{proof}

Similarly, we have the following theorem.
\begin{theorem} \label{ep-a=aa-2}
Let $a\in R^{\dagger}$. Then the following are equivalent:
\begin{itemize}
\item[{\rm (1)}] $a\in R^{\mathrm{EP}}$;
\item[{\rm (2)}] $Ra=Ra^{2}$ and $[aa^{\dagger},a^{\dagger}]=0$;
\item[{\rm (3)}] $Ra=Ra^{2}$ and $[aa^{\dagger},a]=0$;
\item[{\rm (4)}] $Ra=Ra^{2}$ and $Ra\subseteq Ra^{\dagger}$;
\item[{\rm (5)}] $Ra=Ra^{2}$ and $Ra\subseteq Ra^{\ast}$.
\end{itemize}
\end{theorem}

\centerline {\bf ACKNOWLEDGMENTS} This research is supported by the National Natural Science Foundation of China (No. 11201063 and No. 11371089), the Specialized Research Fund for the Doctoral Program of Higher Education (No. 20120092110020); the Jiangsu Planned Projects for Postdoctoral Research Funds (No. 1501048B); the Natural Science Foundation of Jiangsu Province (No. BK20141327).

\end{document}